\documentclass{amsart}

\usepackage{graphics}
\usepackage{epsfig}
\usepackage{amsmath}
\usepackage{amssymb,fancyhdr,txfonts,pxfonts}

\usepackage{paralist}

\usepackage{pstricks,pst-text,pst-grad,pst-node,pst-3dplot,pstricks-add,pst-poly,pst-coil} 
\usepackage{pst-fun,pst-blur} 

\usepackage{hyperref}
\hypersetup{linktocpage=true,colorlinks=true,linkcolor=blue,citecolor=blue,pdfstartview={XYZ 1000 1000 1}}

\newtheorem{Theorem}{\quad Theorem}[section]

\newtheorem{Lemma}[Theorem]{\quad Lemma} 

\newtheorem{example}[Theorem]{Example} 
\newtheorem{proposition}{Proposition}

\usepackage{graphicx}
\makeatletter
\DeclareFontFamily{U}{tipa}{}
\DeclareFontShape{U}{tipa}{bx}{n}{<->tipabx10}{}
\newcommand{\arc@char}{{\usefont{U}{tipa}{bx}{n}\symbol{62}}}%

\newcommand{\arc}[1]{\mathpalette\arc@arc{#1}}
\newcommand{\abs}[1]{\left|#1\right|}     

\newcommand{\arc@arc}[2]{%
  \sbox0{$\m@th#1#2$}%
  \vbox{
    \hbox{\resizebox{\wd0}{\height}{\arc@char}}
    \nointerlineskip
    \box0
  }%
}
\makeatother

\begin{document}

\title[Isomorphisms-Dilations of the skew-fields]{Isomorphic-Dilations of the skew-fields constructed\\ over parallel lines in the Desargues affine plane}

\author[Orgest ZAKA]{Orgest ZAKA}
\address{Orgest ZAKA: 
Department of Mathematics, Faculty of Technical Science,
University of Vlora “Ismail QEMALI”, Vlora, Albania}
\email{orgest.zaka@univlora.edu.al, gertizaka@yahoo.com}

\author[James F. Peters]{James F. Peters}
\address{James F. PETERS: 
Computational Intelligence Laboratory,
University of Manitoba, WPG, MB, R3T 5V6, Canada and
Department of Mathematics, Faculty of Arts and Sciences, Ad\.{i}yaman University, 02040 Ad\.{i}yaman, Turkey}
\thanks{The research has been supported by the Natural Sciences \&
Engineering Research Council of Canada (NSERC) discovery grant 185986, Instituto Nazionale di Alta Matematica (INdAM) Francesco Severi, Gruppo Nazionale  per le Strutture Algebriche, Geometriche e Loro Applicazioni grant 9 920160 000362, n.prot U 2016/000036 and Scientific and Technological Research Council of Turkey (T\"{U}B\.{I}TAK) Scientific Human
Resources Development (BIDEB) under grant no: 2221-1059B211301223.}
\email{James.Peters3@umanitoba.ca}

\dedicatory{Dedicated to  David Hilbert and Girard Desargues}


\begin{abstract}
This paper considers dilations and translations of lines in the Desargues affine plane.   A \emph{dilation} of a line transforms each line into a parallel line
whose length is a multiple of the length of the original line.   In addition to the usual Playfair axiom for parallel lines in an affine plane, further conditions are given for distinct lines to be parallel in the Desargues affine plane.   This paper introduces the dilation of parallel lines in a finite Desargues affine plane that is a bijection of the lines.  Two main results are given in this paper, namely,  each dilation in a finite Desarguesian plane is an isomorphism between skew fields constructed over isomorphic lines and each dilation in a finite Desarguesian plane occurs in a Pappian space.
\end{abstract}

\maketitle

\section{Introduction} 
The foundations for the study of the connections between axiomatic geometry and algebraic structures were
set forth by Hilbert~\cite{Hilbert1959geometry}, recently elaborated and extended in terms of the algebra of affine planes in, for example,~\cite{Kryftis2015thesis},~\cite[\S IX.3, p. 574]{Berger2010geometryRevealed}(affine transformation of the plane transforms a lattice into a lattice) and in the Desargues affine plane (see, {\em e.g.},~\cite{1,2,3,4,5,7,8}).
Properties of geometric structures such as adjacency and proximity~\cite[\S III.XVIII, 193]{PSAleksandrov1956affineGeometry}~\cite{Peters2016AMSJphysicalGeometry,Peters2018JMSMvortexNerves} are preserved by affine transformations~\cite[\S III.XVII, 130ff]{ADAleksandrov1956nonEuclideanGeometry}. 
This paper focuses on the algebra arising from dilatations and translations entirely in the Desargues affine
plane.

In this paper, we consider dilatations and translations entirely in the Desargues affine
plane.  We highlight the case when for a fixed
point of dilatations it is not in any of the lines on which we have
constructed the skew-field (the case when the fixed point is the point of
the line where we constructed the skew-field is exhausted in the paper \cite%
{1}). 

In previous work~\cite{2}, we have shown how we can transform
a line in the Desargues affine plane into an additive Group of its points. 
We have also shown~\cite{3} how to construct a skew-field with a set of
points on a line in the Desargues affine plane.  In addition, for a line of in any 
Desargues affine plane, we construct a skew-field with the
points of this line, by appropriately defined addition and multiplication of points in a line.

The main result in this paper is that each dilation is an isomorphism between skew fields 
constructed over isomorphic lines (see Theorem~\ref{thm:mainResult})..



\begin{figure}[!ht]
\centering
\begin{pspicture}
(-0.5,-1.8)(4,3)
\psline[linestyle=solid](0,-1.5)(0,2.5)\psline[linestyle=solid](2,-1.5)(2,2.5)
\psline[linestyle=solid](4,-1.5)(4,2.5)
\psline[linestyle=solid](0,-1)(2,0)(4,-1) 
\psline[linestyle=dotted, , linewidth=1.2pt,linecolor = red](0,1)(4,1) 
\psline[linestyle=solid](0,1)(2,2)(4,1) 
\psline[linestyle=dotted, , linewidth=1.2pt,linecolor = red](0,-1)(4,-1) 
\psdots[dotstyle=o, linewidth=1.2pt,linecolor = black, fillcolor = yellow]%
(0,1)(0,-1)(2,2)(2,1)(2,0)(2,-1)(4,1)(4,-1)
\rput(0,2.7){$\boldsymbol{\ell_k}$}\rput(2,2.7){$\boldsymbol{\ell_l}$}\rput(4.0,2.7){$\boldsymbol{\ell_m}$}
\rput(-0.25,1){$\boldsymbol{A'}$}\rput(2.25,2.25){$\boldsymbol{B'}$}\rput(4.25,1){$\boldsymbol{C'}$}
\rput(-0.25,-1){$\boldsymbol{A}$}\rput(2.25,0.25){$\boldsymbol{B}$}\rput(4.25,-1){$\boldsymbol{C}$}
\rput(1,-0.2){$\boldsymbol{\ell_{AB}}$}\rput(1,1.8){$\boldsymbol{\ell_{A'B'}}$}
\rput(3,1.8){$\boldsymbol{\ell_{B'C'}}$}\rput(3,-0.2){$\boldsymbol{\ell_{BC}}$}
\rput(2.45,-0.8){$\boldsymbol{\ell_{AC}}$}\rput(2.55,1.3){$\boldsymbol{\ell_{A'C'}}$}
\end{pspicture}
\caption[]{Parallel lines in Desargues Axiom: $\boldsymbol{\ell_{AC}\parallel \ell_{A'C'}}$}
\label{fig:1-path}
\end{figure}
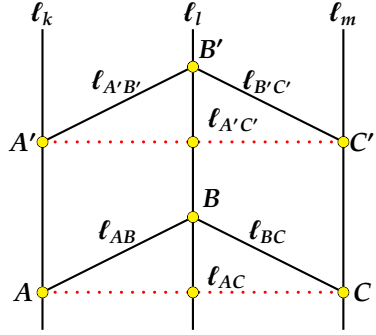

\section{Preliminaties}
Let $\mathcal{P}$ be a nonempty space and let $2^{\pi}$ be a collection of subsets of $\mathcal{P}$ (denoted by $\mathcal{L}$).
The elements $p$ of $\mathcal{P}$ are points and an element $\ell$ of $\mathcal{L}$ is a line.   We write
$P\in \mathcal{P}$ and $P\in \ell$, interchangeably.    The geometric structure $\left(\mathcal{P},\mathcal{L}\right)$ is an \emph{\bf affine plane}, provided\\

\begin{compactenum}[1$^o$]
\item For each $\left\{P,Q\right\}\in \mathcal{P}$, there is exactly one $\ell\in \mathcal{L}$ such that $\left\{P,Q\right\}\in \ell$.

\item For each $P\in \mathcal{P}, \ell\in \mathcal{L}, P \not\in \ell$, there is exactly one $\ell'\in \mathcal{L}$ such that
$P\in \ell'$ and $\ell\cap \ell' = \emptyset$\ (Playfair Parallel Axiom~\cite{Pickert1973PlayfairAxiom}).   Put another way,
if $P\not\in \ell$, then there is a unique line $\ell'$ on $P$ missing $\ell$~\cite{Prazmowska2004DemoMathDesparguesAxiom}.

\item There is a 3-subset $\left\{P,Q,R\right\}\in \mathcal{P}$, which is not a subset of any $\ell$ in the plane.   Put another way,
there exist three non-collinear points $\mathcal{P}$~\cite{Prazmowska2004DemoMathDesparguesAxiom}.
\end{compactenum}

\emph{\bf Desargues' Axiom}~\cite[\S 3.9, pp. 60-61]            {Kryftis2015thesis}~\cite{Szmielew1981DesarguesAxiom}.   Let $A,B,C,A',B',C'\in \mathcal{P}$ and let pairwise distinct lines  $\ell_k , \ell_l, \ell_m, \ell_{AC}, \ell_{A'C'}\in \mathcal{L}$ such that
\begin{align*}
\ell_k \parallel \ell_l \parallel \ell_m &\ \mbox{and}\  \ell_A\parallel \ell_{A'}\ \mbox{and}\ \ell_{C}\parallel \ell_{C'}.\\
A,B\in \ell_{AB}, A'B'\in \ell_{A'B'}, &\ \mbox{and}\ B,C\in \ell_{BC}, B'C'\in \ell_{B'C'}.\\
A\neq C, A'\neq C', &\ \mbox{and}\ \ell_{AB}\neq \ell_{l}, \ell_{BC}\neq \ell_{l}.
\end{align*}
Then $\boldsymbol{\ell_{AC}\parallel \ell_{A'C'}}$.   \qquad \textcolor{blue}{$\blacksquare$}

\begin{example}
The parallel lines  $\ell_{AC}, \ell_{A'C'}\in \mathcal{L}$ in Desargues' Axiom are represented in Fig.~\ref{fig:1-path}.  In other words, the base of $\bigtriangleup ABC$ is parallel with the base of $\bigtriangleup A'B'C'$, provided the restrictions on the points and lines in Desargues' Axiom are satisfied.
\qquad \textcolor{blue}{$\blacksquare$}
\end{example}

\noindent A {\bf Desargues affine plane} is an affine plane that satisfies Desargues' Axiom. 

An \emph{\bf affine space} $\mathcal{P}$ is a nonempty set endowed with structure by the prescription of a real vector space $\vec{\mathcal{P}}$ together with an injective and transitive action of the additive group of $\vec{\mathcal{P}}$ on $\mathcal{P}$.   The vector space $\vec{\mathcal{P}}$ is the translation space of $\mathcal{P}$~\cite{NollSchaffer1977affineIsomorphisms}. 

Let $\mathcal{P},\mathcal{P}'$ be a pair of affine spaces, with associated vector spaces $\vec{\mathcal{P}},\vec{\mathcal{P}'}$.   A mapping $f: \mathcal{P}\longrightarrow \mathcal{P}'$ is affine~\cite{DolgachevShirokov1995affineMap}, provided there is a linear mapping $\phi:\vec{\mathcal{P}}\longrightarrow \vec{\mathcal{P}'}$ such that
\[
f(P + \vec{l}) = f(P) + \phi(\vec{v})\ \mbox{for all}\ P\in \mathcal{P}, \vec{v}\in \vec{\mathcal{P}}.
\]

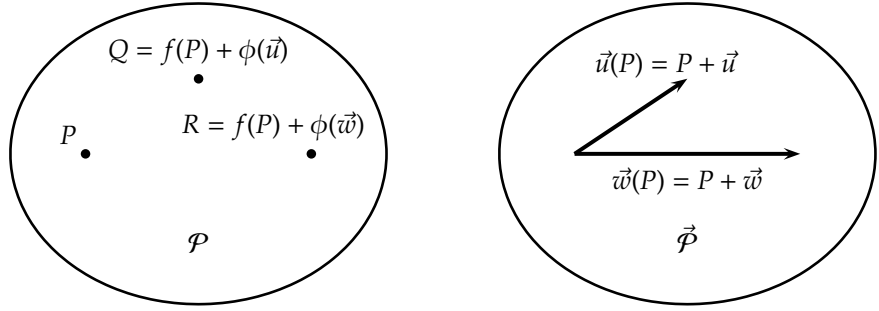
\begin{figure}[!ht]
\centering
    \begin{pspicture}
		(-1,-0.3)(12,4)
    \parametricplot[linewidth=1pt,plotstyle=ccurve]%
    {0}{360}{2.5 t cos 2.5 mul add 2 t sin 2 mul add}
    \parametricplot[linewidth=1pt,plotstyle=ccurve]%
    {0}{360}{9 t cos 2.5 mul add 2 t sin 2 mul add}
    \psline[linewidth=1.5pt]{->}(7.5,2)(10.5,2)
    \psline[linewidth=1.5pt]{->}(7.5,2)(9,3)
    \psdots[dotstyle=*,dotscale=1.0](1,2)
    \psdots[dotstyle=*,dotscale=1.0](2.5,3)
    \psdots[dotstyle=*,dotscale=1.0](4,2)
    \uput[90](2.5,0.5){$\mathcal{P}$}
	  \uput[90](9,0.5){$\vec{\mathcal{P}}$}
    \uput[135](1,2){$P$}
    \uput[90](2.5,3){$Q = f(P) + \phi(\vec{u})$}
    \uput[90](3.5,2){$R = f(P) + \phi(\vec{w})$}
    \uput[0](7.6,3.2){$\vec{u}(P)=P+\vec{u}$}
    \uput[-90](9,2){$\vec{w}(P)=P+\vec{w}$}
    \end{pspicture}
	\caption[]{Affine plane $\mathcal{P}\longmapsto\ \mbox{(associated vector space)}\ \vec{\mathcal{P}}$}
	\label{fig:affineMap}
\end{figure}

\begin{example}
Let $P\in \mathcal{P}, \vec{u}\in \vec{\mathcal{P}}$,  then $\vec{u}(P) = P + \vec{v}$.   A sample affine mapping of $\mathcal{P}$ to $\vec{\mathcal{P}}$ with respect to the point $P$ is shown in Fig.~\ref{fig:affineMap}.  
\qquad \textcolor{blue}{$\blacksquare$}
\end{example}

A space $\left(\mathcal{P},\mathcal{L}\right)$ is Pappian, provided

\begin{Theorem}{\rm [{\bf Pappus}]}.\\
Let $P_1,P_2.P_3, P_4,P_5,P_6\in \mathcal{P}$ and $\abs{\ell_{P_1P_3}\cap \ell_{P_2P_3}} = 1$, then 
$\ell_{P_3Q_4}\parallel \ell_{P_1P_6}$
\end{Theorem}

Every Desarguesian affine plane is isomorphic to a coordinate plane over a field~\cite{Wedderburn1987PappusDesarguesianAffinePlanes} and every finite field is commutative~\cite[\S 3, p. 351]{Wedderburn1905DesarguesianAffinePlane} .   From this, we obtain

\begin{Theorem}\label{thm:Tecklenburg}{\rm [{\bf Tecklenburg}]}{\rm~\cite{Wedderburn1987PappusDesarguesianAffinePlanes}}.\\
Every finite Desarguesian affine plane is Pappian.
\end{Theorem}

%


\section{Main Result: Dilations of Lines in a Finite Desarguesian Plan as Isomorphisms of Skew-fields Constructed over the Lines}

In this section, we will give the main result in this article.

\begin{proposition}\label{prop:bijectionOfLines} {\cite{4}} 
For every two parallel lines $\ell _{1},\ell _{2}\in \mathcal{L}$
in Desargues affine plane, exist a dilatation in this plane, $\delta :\ell
_{1}\longrightarrow \ell _{2}.$ This dilatation $\delta $ is a bijection of
lines $\ell _{1}$ and $\ell _{2}.$
\end{proposition}

\begin{Lemma}\label{lemma:one}
Each dilation, with a fixed point, in a finite Desargues affine plane $\mathcal{A}_{\mathcal{D}}=(%
\mathcal{P},\mathcal{L},\mathcal{I})$ is an isomorphism's between
skew-fields constructed over isomorphic lines $\ell _{1},\ell _{2}\in 
\mathcal{L}$\ of that plane.
\end{Lemma}
\begin{proof}
Consider a dilatation $\delta$ with a fixed point $V$ in a finite Desargues affine plane. The case when this dilatation leads us to a
line in itself is covered in~\cite{1}, where the fixed point for dilatation is on the lines.


Consider a dilation $\delta :\ell _{1}\longrightarrow \ell _{2}$ and a
fixed point $V$ related to dilatation $\delta $ of the Desargues affine plane
such that $V\notin \ell _{1}$ and \ $V\notin \ell _{2}.$  From properties of
dilatations, we know that \ $\delta \left( \ell _{1}\right) =\ell _{2}\parallel
\ell _{1}.$ \ We mark \ $\mathbf{K}_{1}=$\ $\left( \ell _{1},+,\ast \right) $%
\ the skew-field constructed over the line $\ell _{1}$ and $\mathbf{K}_{2}=$%
\ $\left( \ell _{2},+,\ast \right) $\ the skew-field constructed over the
line $\ell _{2}.$

From Prop.~\ref{prop:bijectionOfLines}, we
have that \ $\delta :\ell _{1}\longrightarrow
\ell _{2}$ is a bijection. \ Now we see dilation $\delta $ as a function \ 
$\delta :\mathbf{K}_{1}=\ \left( \ell _{1},+,\ast \right) \longrightarrow 
\mathbf{K}_{2}=$\ $\left( \ell _{2},+,\ast \right) $ namely, as a function \ 
$\delta :\mathbf{K}_{1}\longrightarrow \mathbf{K}_{2}.$ Let's show that $%
\delta $ \ is an isomorphism of these two skew-fields. \ The dilation $%
\delta $ is a bijection of lines $\ell _{1}$ and $\ell _{2}.$ Let us prove
that $\delta $, is a homomorphism between the skew-fields \ $\mathbf{K}_{1}\ 
$and $\mathbf{K}_{2},$ so we have to show that:%

$$
\forall A,C\in \ell _{1},%
\begin{array}{l}
\mathbf{1.}\delta \left( A+C\right) =\delta \left( A\right) +\delta \left(
C\right). \\ 
\mbox{} \\
\mathbf{2}.\delta \left( A\ast C\right) =\delta \left( A\right) \ast \delta
\left( C\right).%
\end{array}%
$$
$\mathbf{Case.1.}$. For points $A,C\in \ell _{1},$ have $\delta \left(
A\right) ,\delta \left( C\right) \in \ell _{2}.$We show firstly that: 
$$
\delta \left( A+C\right) =\delta \left( A\right) +\delta \left( C\right) .
$$
During the construction of the skew-field over a line of Desargues affine
plane, we choose two points (each affine plan has at least two points), which
we write with $\mathbf{O}$ and $\mathbf{I}$ and call them \textbf{zero} and 
\textbf{one} respectively, these points play the role of \textbf{unitary
elements }regarding the two actions addition and multiplication
respectively. In our case, the role of zero is the fixed point $O$. From the
addition algorithm (Algorithm 1 in \cite{2} and \cite{3}) we have:%

$$
\forall A,C\in \ell _{1},\left[ 
\begin{array}{l}
\mathbf{Step.1.}\exists B\notin OI \\ 
\mbox{}\\
\mathbf{Step.2.\ell }_{OI}^{B}\cap \ell _{OB}^{A}=D \\ 
\mbox{}\\
\mathbf{Step.3.\ell }_{CB}^{D}\cap OI=E%
\end{array}%
\right] \Longleftrightarrow A+C=E.
$$
From the construction of the $A+C$ point, we have that: \ 
$$
\ell ^{OB}\parallel \ell ^{AD};\ell ^{BC}\parallel \ell ^{D(A+C)},
$$

and 
$$
O,A,C,A+C\in \ell _{1};B,D\in \ell _{AC}^{B}.
$$

Hence,%
\begin{equation}
D(A+C)\parallel BC.
\end{equation}


We take now the dilatation $\delta $. Mark $\delta (O)=O",$ $\delta
(A)=A",\delta (C)=C",\delta (B)=B"$ we also have that $\delta (V)=V$
(see Fig.~\ref{fig:dilation}).

\begin{figure}[!ht]
\centering
\begin{pspicture}
(-3,-4)(9,6)
\newrgbcolor{wqwqwq}{0.3764705882352941 0.3764705882352941 0.3764705882352941}
\newrgbcolor{eqeqeq}{0.8784313725490196 0.8784313725490196 0.8784313725490196}
\newrgbcolor{ffxfqq}{1. 0.4980392156862745 0.}
\newrgbcolor{xdxdff}{0.49019607843137253 0.49019607843137253 1.}
\newrgbcolor{ffdxqq}{1. 0.8431372549019608 0.}
\newrgbcolor{zzttqq}{0.6 0.2 0.}
\newrgbcolor{dcrutc}{0.8627450980392157 0.0784313725490196 0.23529411764705882}
\psset{xunit=1.0cm,yunit=1.0cm,algebraic=true,dimen=middle,dotstyle=o,dotsize=5pt 0,linewidth=1.6pt,arrowsize=3pt 2,arrowinset=0.25}
\psline[linewidth=1.2pt,linecolor=wqwqwq](-1.,5.14)(7.8,5.14) 
\psline[linewidth=1.5pt](-2.,2.)(7.5,2.) 
\psline[linewidth=1.5pt](-1.,-1.)(6.,-1.) 
\psline[linewidth=1.5pt](1.38,-3.36)(-2.,2.)
\psline[linewidth=1.5pt](1.38,-3.36)(0.36,2.)
\psline[linewidth=1.5pt](1.38,-3.36)(3.2,2.)
\psline[linewidth=1.5pt,linecolor=wqwqwq](-1.58,0.6)(6.06,0.6) 
\psline[linewidth=2.pt,linestyle=dashed,dash=2pt 2pt,linecolor=blue](-0.10820895522388058,-1.)(1.36,0.6)
\psline[linewidth=2.pt,linestyle=dashed,dash=2pt 2pt,linecolor=blue](0.9317910447761196,-1.)(2.38,0.6)
\psline[linewidth=2.pt,linestyle=dashed,dash=4pt 4pt,linecolor=zzttqq](1.36,0.6)(2.18134328358209,-1.)
\psline[linewidth=2.pt,linestyle=dashed,dash=4pt 4pt,linecolor=zzttqq](2.38,0.6)(3.2013432835820894,-1.)
\psline[linewidth=2.pt](1.38,-3.36)(5.56,2.)
\psline[linewidth=2.pt,linestyle=dashed,dash=2pt 2pt,linecolor=blue](-2.,2.)(1.36,5.14)
\psline[linewidth=2.pt,linestyle=dashed,dash=2pt 2pt,linecolor=blue](0.36,2.)(3.5,5.14)
\psline[linewidth=2.pt,linestyle=dashed,dash=2pt 2pt,linecolor=dcrutc](3.2,2.)(1.36,5.14)
\psline[linewidth=2.pt,linestyle=dashed,dash=2pt 2pt,linecolor=dcrutc](5.56,2.)(3.5,5.14)
\psline[linewidth=2.pt,linestyle=dotted](1.38,-3.36)(1.36,5.14)
\psline[linewidth=2.pt,linestyle=dotted](1.38,-3.36)(3.5,5.14)
\rput[tl](1.54,-3.26){$V$}
\rput[tl](-0.52,-1.2){$O$}\rput[tl](0.46,-1.2){$A$}\rput[tl](2.22,-1.2){$C$}\rput[tl](3.32,-1.2){$A+C$}
\rput[tl](6.24,-0.9){${\ell}_1 $}
\rput[tl](-2.56,2.4){$O''$}\rput[tl](-0.3,2.41){$A''$}\rput[tl](3.36,2.41){$C''$}\rput[tl](5.68,2.41){$A''+C''$}
\rput[tl](7.8,2.31){${\ell}_2=\delta ({\ell}_1)$}
\rput[tl](5.68,1.85){$\delta (A+C)=A''+C''$}
\rput[tl](4.86,5.74){${\ell}_{A''C''}^{B''}={\ell}_{{\ell}_{1}}^{B''}=\delta ({\ell}_{{\ell}_1}^B)$}
\rput[tl](1.24,5.54){$B''$}\rput[tl](3.48,5.54){$D''$}
\rput[tl](1.05,0.95){$B$}\rput[tl](2.05,0.95){$D$}
\rput[tl](6.1,0.95){${\ell}_{AC}^B={\ell}_{{\ell}_1}^B$}
\begin{scriptsize}
\psdots[dotstyle=o, linewidth=1.2pt,linecolor = black, fillcolor = lightgray](1.38,-3.36)
\psdots[dotstyle=o, linewidth=1.2pt,linecolor = black, fillcolor = red](-0.10820895522388058,-1.)
\psdots[dotstyle=o, linewidth=1.2pt,linecolor = black, fillcolor = red](2.18134328358209,-1.)
\psdots[dotstyle=o, linewidth=1.2pt,linecolor = black, fillcolor = red](0.9317910447761196,-1.)
\psdots[dotstyle=o, linewidth=1.2pt,linecolor = black, fillcolor = red](3.2013432835820894,-1.)
\psdots[dotstyle=o, linewidth=1.2pt,linecolor = black, fillcolor = yellow](1.36,0.6)
\psdots[dotstyle=o, linewidth=1.2pt,linecolor = black, fillcolor = yellow](2.38,0.6)
\psdots[dotstyle=o, linewidth=1.2pt,linecolor = black, fillcolor = green](-1.9882089552238809,2.02)
\psdots[dotstyle=o, linewidth=1.2pt,linecolor = black, fillcolor = green](0.36,2.)
\psdots[dotstyle=o, linewidth=1.2pt,linecolor = black, fillcolor = green](3.2,2.)
\psdots[dotstyle=o, linewidth=1.2pt,linecolor = black, fillcolor = green](5.56,2.)
\psdots[dotstyle=o, linewidth=1.2pt,linecolor = black, fillcolor = blue](1.36,5.14)
\psdots[dotstyle=o, linewidth=1.2pt,linecolor = black, fillcolor = blue](3.5,5.14)
\end{scriptsize}
\end{pspicture}
\caption{The addition of two points in a line of Desargues affine plane under Dilation: $\delta (A + C) = \delta (A) + \delta (C).
$}
\label{fig:dilation}
\end{figure}

By dilation properties (see \cite{4}), we have that: \ 
$$
O"\in \ell ^{OV},B"\in \ell ^{VB},C"\in \ell ^{VC},A"\in \ell ^{VA},
$$

also from the definition of dilation (see \cite{4}), we have that:%
$$
OB\parallel \delta (O)\delta (B)=O"B";BC\parallel \delta (B)\delta (C)=B"C".
$$

Calculate now $\delta (A)+\delta (C)$ according to the addition algorithm
(see \cite{2} and \cite{3}). During addition of points $A$ and $C$ we chose
the point $B\notin \ell _{1},$ here we choose the point $\delta (B)=B"\notin
\ell _{2}$ (we do this for ease of proof, as in \cite{2}, we have shown that
the choice of auxiliary point $B$, or in our case $\delta (B)=B"$ is
arbitrary), by applying the addition algorithm we have:%
$$
\forall \delta (A),\delta (C)\in \ell _{2},\left[ 
\begin{array}{l}
\mathbf{Step.1.}\exists B"\notin \ell _{2} \\ 
\mbox{}\\
\mathbf{Step.2.\ell }_{\ell _{2}}^{B"}\cap \ell _{O"B"}^{A"}=D" \\
\mbox{}\\ 
\mathbf{Step.3.\ell }_{C"B"}^{D"}\cap \ell _{2}=E"%
\end{array}%
\right] \Longleftrightarrow A"+C"=E".
$$


\noindent From this we have:\\

\begin{eqnarray*}
\left. 
\begin{array}{@{\hspace{.5em}}c}
{A"}+C"\in \ell _{B"C"}^{D"}=\ell _{D(A+C)}^{D"} \\ 
{A"}+C"\in \ell _{2}%
\end{array}%
\right\}  &\Longrightarrow &A"+C"=\ell _{D(A+C)}^{D"}\cap \ell _{2} \\
\Longrightarrow \delta \left( A\right) +\delta \left( C\right)  &=&\ell_{D(A+C)}^{D"}\cap \ell _{2}.
\end{eqnarray*}


On the other hand, we have
\[
\delta :\ell _{1}\rightarrow \ell _{2}\ \mbox{and}
\]
\[
\left.
\begin{array}{r@{\hspace{.5em}}cc}
D&\mapsto D" = \delta \left( D\right)  \\
A+C & \mapsto \delta \left( A+C\right)
\end{array}
\right\}  \Longrightarrow D(A+C)\parallel D"\delta \left( A+C\right) = \delta\left( D\right) \delta \left( A+C\right).
\]

From this result, have that 

\[
\left.
\begin{array}{r@{\hspace{.5em}}cc}
\delta \left( A+C\right) &\in \ell _{D(A+C)}^{D"}  \\
\delta \left( D\right) ,\delta \left( A+C\right) & \in \ell _{2}
\end{array}
\right\}  \Rightarrow \delta \left( A+C\right) =\ell _{D(A+C)}^{D"}\cap \ell
_{2}=A"+C"=\delta \left( A\right) +\delta \left( C\right) .
\]


Hence, we achieve the required result, namely,%
$$
\delta \left( A+C\right) =\delta \left( A\right) +\delta \left( C\right) .
$$

This proof can also be made in the three-vertexes language (see \cite{5})

From the above constructions and results, we have a similarity between the
two three-vertexes 
$$ (O,C,B)\thicksim(O",C",B"),$$

we also have the similarity of 
$$
(A,A+C,D)\thicksim (A",A"+C",D").
$$

But from the similarity of the three-vertexes in the Desargues affine plane,
we also have: 
$$
(A,A+C,D)\thicksim (A",\delta \left( A+C\right) ,D")
$$

or in other words 
$$
(A,A+C,D)\thicksim (\delta \left( A\right) ,\delta \left( A+C\right) ,\delta
\left( D\right) ),
$$

from these we have 
$$
(A",A"+C",D")\thicksim (A",\delta \left( A+C\right) ,D").
$$

Since these two similar three-vertexes have the same two vertices then the
third vertice it will be the same, thus 
$$
A"+C"=\delta \left( A+C\right) \Longrightarrow \delta \left( A+C\right)
=\delta \left( A\right) +\delta \left( C\right)
$$

In parallelograms language, we would say from the above parallels that:
$$
\left( O,A,D,B\right) \thicksim \left( O",A",D",B"\right)
$$

and 
$$(C,A+C,D,B)\thicksim (\delta \left( C\right) ,\delta \left( A+C\right)
,\delta \left( D\right) ,\delta \left( B\right) )=(C",\delta \left(
A+C\right) ,D",B")$$

But on the other hand, from the above parallelisms and from the addition of
points $A"$ and $C"$ we would have the similarity of the parallelograms: 

$$ (C,A+C,D,B)\thicksim (C",A"+C",D",B"),$$

since the similarity of the $n$-vertexes in the Desargues affine plane is the
relation of equivalence (see \cite{5}) (such would be and for
parallelograms), we have: 
$$ (C",\delta \left( A+C\right) ,D",B")\thicksim (C",A"+C",D",B"). $$

or, equally

$$(\delta \left( C\right) ,\delta \left( A+C\right) ,\delta \left( D\right)
,\delta \left( B\right) )\thicksim (\delta \left( C\right) ,\delta \left(
A\right) +\delta \left( C\right) ,\delta \left( D\right) ,\delta \left(
B\right) ).$$

Since these two parallel parallelograms have three equal vertices they would
be equal, then the fourth vertices would coincide, thus%

$$ (C",\delta \left( A+C\right) ,D",B")=(C",A"+C",D",B")\Longrightarrow\\
 \delta\left( A+C\right) =\delta \left( A\right) +\delta \left( C\right) . $$

$\mathbf{Case.2.}$ We see now that stands for multiplication properties: 
$$ \delta \left( A\ast C\right) =\delta \left( A\right) \ast \delta \left(
C\right)$$

During the multiplication of points $A$ and $C,$ we choose a point $I$ (necessarily different from the point $O$ that we have assigned in the role of the element 'zero'. We do this freely, as each line in affine plan having at least 2-points) in
the role of "\textbf{one}" (see Fig.~\ref{fig:affineMap}), the dilatation $\delta $\ \ gives us
the point \textbf{I} in point $\delta (I)=I",$ then from point $A$ we
construct the line \ $\ell _{BC}^{E}$ and thus determine multiplication as
well:%
\begin{equation}
A\ast C=\ell _{BC}^{E}\cap \ell _{1}.
\end{equation}

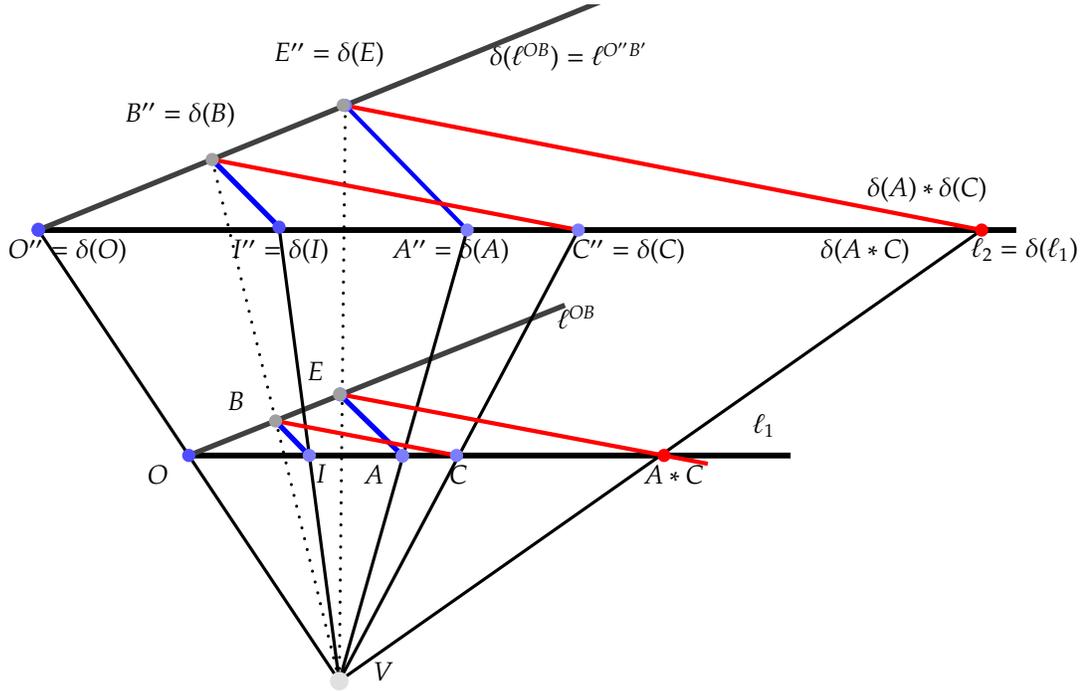
\begin{figure}[!ht]
\centering%
\newrgbcolor{ududff}{0.30196078431372547 0.30196078431372547 1.}
\newrgbcolor{uququq}{0.25098039215686274 0.25098039215686274 0.25098039215686274}
\newrgbcolor{eqeqeq}{0.8784313725490196 0.8784313725490196 0.8784313725490196}
\newrgbcolor{xdxdff}{0.49019607843137253 0.49019607843137253 1.}
\newrgbcolor{aqaqaq}{0.6274509803921569 0.6274509803921569 0.6274509803921569}
\psset{xunit=1.0cm,yunit=1.0cm,algebraic=true,dimen=middle,dotstyle=o,dotsize=5pt 0,linewidth=1.6pt,arrowsize=3pt 2,arrowinset=0.25}
\begin{pspicture*}(-2.4,-6.5)(13.,3.)
\psline[linewidth=2.2pt](0.,-3.)(8.,-3.)
\psline[linewidth=2.pt,linecolor=uququq](0.,-3.)(5.,-1.)
\psline[linewidth=1.2pt](2.,-6.)(-2.,0.)
\psline[linewidth=2.2pt](-2.,0.)(11.,0.)
\psline[linewidth=2.pt,linecolor=uququq](-2.,0.)(6.14,3.31)
\psline[linewidth=2.pt,linecolor=blue](1.6,-3.)(1.1517241379310348,-2.539310344827586)
\psline[linewidth=2.pt,linecolor=blue](2.84,-3.)(2.013793103448276,-2.1944827586206896)
\psline[linewidth=1.2pt,linestyle=dotted](2.,-6.)(0.3089180309185826,0.9388843590098904)
\psline[linewidth=1.2pt](2.,-6.)(1.2,0.04)
\psline[linewidth=1.2pt](2.,-6.)(3.7,0.)
\psline[linewidth=1.2pt](2.,-6.)(5.18,0.)
\psline[linewidth=2.pt,linecolor=blue](1.2,0.04)(0.3089180309185826,0.9388843590098904)
\psline[linewidth=1.2pt,linestyle=dotted](2.,-6.)(2.080610419901652,1.659314556495635)
\psline[linewidth=1.6pt,linecolor=blue](2.080610419901652,1.659314556495635)(3.7,0.)
\psline[linewidth=1.6pt,linecolor=red](1.1517241379310348,-2.539310344827586)(3.56,-3.)
\psline[linewidth=1.6pt,linecolor=red](0.3089180309185826,0.9388843590098904)(5.18,0.)
\psline[linewidth=1.6pt,linecolor=red](2.0019392947289205,-2.1839534317502793)(6.9,-3.11)
\psline[linewidth=1.2pt](2.,-6.)(10.54,0.)
\psline[linewidth=1.6pt,linecolor=red](2.054473610936636,1.6611434462162489)(10.54,0.)
\rput[tl](2.45,-5.75){$V$}
\rput[tl](-0.55,-3.12){$O$}
\rput[tl](-2.4,-0.12){$O''=\delta(O)$}
\rput[tl](1.7,-3.12){$I$}
\rput[tl](2.34,-3.12){$A$}
\rput[tl](3.47,-3.12){$C$}
\rput[tl](0.58,-0.12){$I''=\delta(I)$}
\rput[tl](2.72,-0.12){$A''=\delta(A)$}
\rput[tl](5.1,-0.12){$C''=\delta(C)$}
\rput[tl](0.52,-2.15){$B$}
\rput[tl](-0.84,1.7){$B''=\delta(B)$}
\rput[tl](1.58,-1.75){$E$}
\rput[tl](1.14,2.5){$E''=\delta(E)$}
\rput[tl](7.5,-2.45){$\ell_{1}$}
\rput[tl](10.4,-0.1){$\ell_{2}=\delta(\ell_{1})$}
\rput[tl](4.9,-1.){$\ell^{OB}$}
\rput[tl](4.,2.5){$\delta(\ell^{OB})=\ell^{O''B'}$}
\rput[tl](6.06,-3.12){$A*C$}
\rput[tl](8.4,-0.12){$\delta(A*C)$}
\rput[tl](9.02,0.7){$\delta(A)*\delta(C)$}
\begin{scriptsize}
\psdots[dotstyle=*,linecolor=ududff](0.,-3.)
\psdots[dotsize=7pt 0,dotstyle=*,linecolor=eqeqeq](2.,-6.)
\psdots[dotstyle=*,linecolor=ududff](-2.,0.)
\psdots[dotstyle=*,linecolor=ududff](-2.,0.01)
\psdots[dotstyle=*,linecolor=ududff](-2.,0.)
\psdots[dotstyle=*,linecolor=xdxdff](1.6,-3.)
\psdots[dotstyle=*,linecolor=xdxdff](2.84,-3.)
\psdots[dotstyle=*,linecolor=xdxdff](3.56,-3.)
\psdots[dotstyle=*,linecolor=aqaqaq](1.1517241379310348,-2.539310344827586)
\psdots[dotstyle=*,linecolor=xdxdff](2.013793103448276,-2.1944827586206896)
\psdots[dotstyle=*,linecolor=aqaqaq](0.3089180309185826,0.9388843590098904)
\psdots[dotstyle=*,linecolor=ududff](1.2,0.04)
\psdots[dotstyle=*,linecolor=xdxdff](3.7,0.)
\psdots[dotstyle=*,linecolor=xdxdff](5.18,0.)
\psdots[dotstyle=*,linecolor=xdxdff](2.84,-3.01)
\psdots[dotstyle=*,linecolor=xdxdff](2.080610419901652,1.659314556495635)
\psdots[dotstyle=*,linecolor=xdxdff](2.014473610936636,-2.1788565537837514)
\psdots[dotstyle=*,linecolor=aqaqaq](2.054473610936636,1.6611434462162489)
\psdots[dotstyle=*,linecolor=aqaqaq](2.0019392947289205,-2.1839534317502793)
\psdots[dotstyle=*,linecolor=red](6.318186194892819,-3.)
\psdots[dotstyle=*,linecolor=red](10.54,0.)
\end{scriptsize}
\end{pspicture*}
\caption{The multiplication of two points in a line of the Desargues affine plane under Dilatation: $ \delta (A * C) =\delta (A) * \delta (C) $}
\label{Figure.2}
\end{figure}

By dilation properties have 
$$
O"\in \ell ^{VO},I"\in \ell ^{VI},C"\in \ell ^{VC},A"\in \ell ^{VA},B"\in
\ell ^{VB},
$$

also for these points we have:

$$
O",A",C"\in \ell _{2}=\delta \left( \ell _{1}\right) \text{ and }B"\in
\delta \left( OB\right) .
$$

Also by definition and dilation properties (see~\cite{3}), we have:  
$$
OB\parallel \delta \left( O\right) \delta \left( B\right) =O"B",BC\parallel
\delta \left( B\right) \delta \left( C\right) =B"C",IB\parallel \delta
\left( I\right) \delta \left( B\right) =I"B"
$$

Calculate now $A"\ast C"$according to multiplication Algorithm (see \cite{2},%
\cite{3}). During the multiplication of points $A$ and $C$ we chose, as
auxiliary points, the point $B\notin \ell _{1},$ here we choose the point $%
B"=\delta \left( B\right) \notin \ell _{2}$ (we do this for ease of proof,
as in \cite{2}, we have shown that the choice of auxiliary point $B$, or in
our case $\delta (B)$ is arbitrary), by applying the multiplication
algorithm we have:
$$
\forall A"=\delta \left( A\right) ,C"=\delta \left( C\right) \in \ell
_{2},\left\{ 
\begin{array}{l}
\mathbf{Step.}1.\exists B"\notin \ell _{2} \\ 
\mbox{} \\ 
\mathbf{Step.}2.\ell _{I"B"}^{A"}\cap O"B"=E" \\ 
\mbox{} \\ 
\mathbf{Step.}3.\ell _{B"C"}^{E"}\cap \ell _{2}=F"%
\end{array}%
\right\} \Longleftrightarrow A"\ast C"=F".
$$

Now multiply by Algorithm 2 of the multiplication definition (see \cite{3})
to obtain
\begin{equation*}
A"=\delta \left( A\right) ,C"=\delta \left( C\right) \in \ell _{2}.
\end{equation*}

From point $A"$ construct the line $\ell _{I"B"}^{A"}$ but since $%
IB\parallel I"B"$, we have that $\ell _{I"B"}^{A"}=\ell _{IB}^{A"}$ and 
$\ell _{I"B"}^{A"}\cap \ell ^{O"B"}=E"$, we have
\begin{equation*}
A"E"\parallel I"B"\parallel IB\parallel AE\Longrightarrow A"E"\parallel AE.
\end{equation*}

Now from point $E"$ construct the line $\ell _{B"C"}^{E"},$ by the
definition of multiplication, we have that:%
\begin{equation*}
E"(A"\ast C")\parallel B"C",
\end{equation*}

but%
\begin{equation*}
\left[ 
\begin{array}{c}
B"C"\parallel BC \\ 
BC\parallel E(A\ast C)%
\end{array}%
\right] \Longrightarrow E"(A"\ast C")\parallel E(A\ast C).
\end{equation*}

See now $\delta \left( E\right) ,$ by dilatation properties have:%
\begin{equation*}
\delta \left( \ell _{1}\right) =\ell _{2},\delta \left( OB\right) =O"B",
\end{equation*}

we consider two three-vertexes $(O,A,E)$ and $(O",A",E"),$ have a similarity
of these two three-vertexes, because:

\begin{equation*}
\begin{array}{cc}
\left. 
\begin{array}{c}
AE\parallel A"E", \\ 
OA\parallel O"A"(\ell _{1}\parallel \ell _{2}), \\ 
OE\parallel O"E" \\ 
(E"\in O"B",E\in OB,OB\parallel O"B")%
\end{array}%
\right\} & \Longrightarrow (O,A,E)\sim (O",A",E")%
\end{array}%
\end{equation*}

Mark with $E^{\prime \prime \prime }=\delta \left( E\right) ,$ by dilatation
properties and from the above parallels, and we have that,%
\begin{equation*}
AE\parallel A"E^{\prime \prime \prime }=A"\delta \left( E\right) \Rightarrow
\delta \left( E\right) \in \ell _{AE}^{A"}=\ell _{A"E"}^{A"}
\Rightarrow \delta \left( E\right) \in \ell _{A"E"}^{A"}
\end{equation*}

but 
\begin{equation*}
\delta \left( E\right) \in O"B"\Rightarrow \delta \left( E\right) =\ell
_{A"E"}^{A"}\cap O"B"=E"\Rightarrow E"=E^{\prime \prime \prime }
\end{equation*}

Or with three-vertexes similarities 
\begin{equation*}
OA\parallel O"A",OE\parallel O"E^{\prime \prime \prime }\Longrightarrow
(O,A,E)\sim (O",A",E^{\prime \prime \prime }),
\end{equation*}

since similarity, in the Desargues affine plane, is a\ equivalence relation,
have that:%
\begin{equation*}
(O",A",E")\sim (O",A",E^{\prime \prime \prime })\Longrightarrow E"=E^{\prime
\prime \prime }\Longrightarrow E"=\delta \left( E\right) .
\end{equation*}

so, by 
\begin{equation*}
E"(A"\ast C")\parallel E(A\ast C)\Rightarrow E(A\ast C)\parallel \delta
\left( E\right) \left[ \delta \left( A\right) \ast \delta \left( C\right) %
\right]
\end{equation*}%
\begin{equation*}
\Rightarrow \delta \left( A\right) \ast \delta \left( C\right) =\ell
_{\delta \left( B\right) \delta \left( C\right) }^{\delta \left( E\right)
}\cap \ell _{2}=\ell _{E\left( A\ast C\right) }^{\delta \left( E\right)
}\cap \ell _{2}.
\end{equation*}

See now, two three-vertexes $(A,E,A\ast C)$ and $(A",E",\delta \left( A\ast
C\right) ),$ from above we have similarity 
\begin{equation*}
(A,E,A\ast C)\sim (A",E",A"\ast C").
\end{equation*}

From dilation properties, have that:%
\begin{equation*}
A"\in \ell ^{VA},E"\in \ell ^{VE},\delta \left( A\ast C\right) \in \ell
^{V\left( A\ast C\right) }
\end{equation*}

and 
\begin{equation*}
AE\parallel A"E",A\left( A\ast C\right) \parallel A"\delta \left( A\ast
C\right) \Longrightarrow \delta \left( A\ast C\right) \in \ell _{2},
\end{equation*}%
\begin{equation*}
E\left( A\ast C\right) \parallel E"\delta \left( A\ast C\right)
\Longrightarrow \delta \left( A\ast C\right) \in \ell _{E\left( A\ast
C\right) }^{E"}
\end{equation*}

Thus, 
\begin{equation*}
\delta \left( A\ast C\right) =\ell _{E\left( A\ast C\right) }^{E"}\cap \ell
_{2}=\delta \left( A\right) \ast \delta \left( C\right) ,
\end{equation*}

so we proved that:%
\begin{equation*}
\delta \left( A\ast C\right) =\delta \left( A\right) \ast \delta \left(
C\right) .
\end{equation*}
\end{proof}

\begin{Lemma}\label{lemma:two}
Each translation (dilatation which is different from $id_{\mathcal{P}}$  and has no fixed point) in a finite Desargues affine plane $\mathcal{A}_{\mathcal{D}}=(%
\mathcal{P}, \mathcal{L},\mathcal{I})$ is an isomorphism's between
skew-fields constructed over isomorphic lines $\ell _{1},\ell _{2}\in 
\mathcal{L},$\ of that plane.
\end{Lemma}

\begin{proof}
Consider a translation $\varphi $ $\ $%
(different from $id_{\mathcal{P}}$ and there is no fixed points ) in the
Desargues affine plane.
The case when this translation leads us
to a line in itself (the translation has a direction according to line $%
\ell $ ), we have dealt with \ \cite{1}.

Consider a translation $\varphi :\ell _{1}\longrightarrow \ell _{2}$ in the
Desargues affine plane. From properties of translations, know that \ $\varphi
\left( \ell _{1}\right) =\ell _{2}\parallel \ell _{1}.$ This translation has
a different direction from that of the lines $\ell _{1},\ell _{2}$ which
belong to the same equivalence class, according to the parallelism.

We mark \ $\mathbf{K}_{1}=$\ $\left( \ell _{1},+,\ast \right) $\ the
skew-field constructed over the line $\ell _{1}$ and $\mathbf{K}_{2}=$\ $%
\left( \ell _{2},+,\ast \right) $\ the skew-field constructed over the line $%
\ell _{2}$ same way as in Lemma~\ref{lemma:one}. 

From Prop.~\ref{prop:bijectionOfLines}, 
we have that \ 
$\varphi :\ell _{1}\longrightarrow \ell _{2}$ is a bijection. \ Now we see
this translations $\varphi $ as a function \ $\varphi :\mathbf{K}_{1}=\
\left( \ell _{1},+,\ast \right) \longrightarrow \mathbf{K}_{2}=$\ $\left(
\ell _{2},+,\ast \right) $ namely, as a function \ $\varphi :\mathbf{K}%
_{1}\longrightarrow \mathbf{K}_{2}.$  Next, we show that $\varphi $ \ is an
isomorphism of these two skew-fields. \  The translation $\varphi $ is a
bijection of lines $\ell _{1}$ and $\ell _{2}.$    Let us prove that $\varphi $%
, is a homomorphism between the skew-fields \ $\mathbf{K}_{1}\ $and $\mathbf{%
K}_{2},$ so we have to show that

\begin{equation*}
\forall A,C\in \ell _{1},%
\begin{array}{l}
\mathbf{1.}\varphi \left( A+C\right) =\varphi \left( A\right) +\varphi
\left( C\right) \\ 
\mbox{} \\ 
\mathbf{2}.\varphi \left( A\ast C\right) =\delta \left( A\right) \ast
\varphi \left( C\right)%
\end{array}%
\end{equation*}

{\bf Case 1}  For points $A,C\in \ell _{1},$ we have $\varphi \left(
A\right) ,\varphi \left( C\right) \in \ell _{2}.$  We show first that: 
\begin{equation*}
\varphi \left( A+C\right) =\varphi \left( A\right) +\varphi \left( C\right) .
\end{equation*}

During the construction of the skew-field over a line on the Desargues affine
plane, we choose two points (each affine plan has at least two points) which
we write with $\mathbf{O}$ and $\mathbf{I}$ and call them \textbf{zero} and 
\textbf{one} respectively, these points play the role of \textbf{unitary
elements }regarding the two actions addition and multiplication
respectively. In our case, the role of zero is the fixed point $O$. From the
addition algorithm (Algorithm 1 in \cite{2} and \cite{3}) we have:%
\begin{equation*}
\forall A,C\in \ell _{1},\left[ 
\begin{array}{l}
\mathbf{Step.1.}\exists B\notin OI \\ 
\mbox{} \\
\mathbf{Step.2.\ell }_{OI}^{B}\cap \ell _{OB}^{A}=D \\ 
\mbox{} \\
\mathbf{Step.3.\ell }_{CB}^{D}\cap OI=E%
\end{array}%
\right] \Longleftrightarrow A+C=E
\end{equation*}

From the construction of the $A+C$ point, we have that: \ 
\begin{equation*}
\ell ^{OB}\parallel \ell ^{AD};\ell ^{BC}\parallel \ell ^{D(A+C)};
\end{equation*}

and 
\begin{equation*}
O,A,C,A+C\in \ell _{1};B,D\in \ell _{AC}^{B}=\ell _{\ell _{1}}^{B},
\end{equation*}

Hence,
\[
D(A+C)\parallel BC.
\]

We take now the translation $\varphi $.   
\begin{equation*}
\varphi (O)=O^{\prime },\varphi (A)=A^{\prime },\varphi (C)=C^{\prime
},\varphi (B)=B^{\prime }.
\end{equation*}

and we have that
\begin{equation*}
A^{\prime }+C^{\prime }=\varphi (A)+\varphi (C)
\end{equation*}

\begin{figure}[!ht]
\centering%
\newrgbcolor{xdxdff}{0.49019607843137253 0.49019607843137253 1.}
\newrgbcolor{aqaqaq}{0.6274509803921569 0.6274509803921569 0.6274509803921569}
\psset{xunit=1.0cm,yunit=1.0cm,algebraic=true,dimen=middle,dotstyle=o,dotsize=5pt 0,linewidth=1.6pt,arrowsize=3pt 2,arrowinset=0.25}
\begin{pspicture*}(0.9,0.3)(11.6,6.8)
\psline[linewidth=2.pt](1.,1.)(10.,1.)
\psline[linewidth=2.pt,linecolor=gray](1.64,2.98)(9.3,3.)
\psline[linewidth=2.pt,linecolor=blue](1.46,1.)(2.840096257413593,2.9831334105937692)
\psline[linewidth=2.pt,linecolor=red](2.840096257413593,2.9831334105937692)(5.,1.)
\psline[linewidth=2.pt,linecolor=blue](3.73,1.02)(5.110096257413593,3.0031334105937693)
\psline[linewidth=2.pt,linecolor=red](5.110096257413593,3.0031334105937693)(7.27,1.02)
\psline[linewidth=2.pt](1.75,3.98)(10.75,3.98)
\psline[linewidth=2.pt,linecolor=gray](2.39,5.96)(10.05,5.98)
\psline[linewidth=2.pt,linecolor=blue](2.21,3.98)(3.590096257413597,5.963133410593769)
\psline[linewidth=2.pt,linecolor=red](3.590096257413597,5.963133410593769)(5.75,3.98)
\psline[linewidth=2.pt,linecolor=blue](4.5,4.)(5.860096257413592,5.983133410593769)
\psline[linewidth=2.pt,linecolor=red](5.860096257413592,5.983133410593769)(8.02,4.)
\psline[linewidth=1.2pt]{->}(1.46,1.)(2.21,3.98)
\psline[linewidth=1.2pt]{->}(3.73,1.02)(4.5,4.)
\psline[linewidth=1.2pt]{->}(5.,1.)(5.75,3.98)
\psline[linewidth=1.2pt]{->}(7.27,1.02)(8.02,4.)
\psline[linewidth=1.2pt,linestyle=dotted]{->}(2.840096257413593,2.9831334105937692)(3.590096257413597,5.963133410593769)
\psline[linewidth=1.2pt,linestyle=dotted]{->}(5.110096257413593,3.0031334105937693)(5.860096257413592,5.983133410593769)
\rput[tl](1.2,0.85){$O$}
\rput[tl](3.44,0.85){$A$}
\rput[tl](4.78,0.85){$C$}
\rput[tl](6.6,0.85){$A+C$}
\rput[tl](2.32,3.4){$B$}
\rput[tl](4.58,3.4){$D$}
\rput[tl](3.46,6.4){$B'$}
\rput[tl](5.8,6.4){$D'$}
\rput[tl](1.52,3.85){$O'$}
\rput[tl](3.8,3.85){$A'$}
\rput[tl](5.92,3.85){$C'$}
\rput[tl](7.68,4.5){$A'+C'=\varphi(A)+\varphi(C)$}
\rput[tl](8.06,3.85){$\varphi(A+C)$}
\rput[tl](8.34,1.52){$\ell_{1}$}
\rput[tl](8.38,2.85){$\ell^{B}_{AC}$}
\rput[tl](7.06,5.83){$\ell^{B'}_{A'C'}=\ell^{B'}_{AC}$}
\rput[tl](9.78,3.85){$\ell_{2}=\varphi(\ell_{1})$}
\rput[tl](1.36,2.7){$\varphi$}
\rput[tl](7.3,2.7){$\varphi$}
\rput[tl](5.64,4.8){$\varphi$}
\rput[tl](3.32,4.8){$\varphi$}
\rput[tl](5.34,2.){$\varphi$}
\rput[tl](3.76,2.7){$\varphi$}
\begin{scriptsize}
\psdots[dotsize=7pt 0,dotstyle=*,linecolor=xdxdff](1.46,1.)
\psdots[dotstyle=*,linecolor=xdxdff](3.74,1.)
\psdots[dotsize=7pt 0,dotstyle=*,linecolor=xdxdff](5.,1.)
\psdots[dotsize=6pt 0,dotstyle=*,linecolor=aqaqaq](2.840096257413593,2.9831334105937692)
\psdots[dotsize=7pt 0,dotstyle=*,linecolor=xdxdff](3.73,1.02)
\psdots[dotstyle=*,linecolor=xdxdff](5.110096257413593,3.0031334105937693)
\psdots[dotsize=7pt 0,dotstyle=*,linecolor=xdxdff](7.27,1.02)
\psdots[dotsize=6pt 0,dotstyle=*,linecolor=aqaqaq](5.110096257413593,3.0031334105937693)
\psdots[dotsize=7pt 0,dotstyle=*,linecolor=xdxdff](2.21,3.98)
\psdots[dotstyle=*,linecolor=xdxdff](4.49,3.98)
\psdots[dotsize=7pt 0,dotstyle=*,linecolor=xdxdff](5.75,3.98)
\psdots[dotsize=6pt 0,dotstyle=*,linecolor=aqaqaq](3.590096257413597,5.963133410593769)
\psdots[dotsize=7pt 0,dotstyle=*,linecolor=xdxdff](4.5,4.)
\psdots[dotstyle=*,linecolor=xdxdff](5.860096257413592,5.983133410593769)
\psdots[dotsize=7pt 0,dotstyle=*,linecolor=xdxdff](8.02,4.)
\psdots[dotsize=6pt 0,dotstyle=*,linecolor=aqaqaq](5.860096257413592,5.983133410593769)
\end{scriptsize}
\end{pspicture*}
\caption{The addition of two points in a line of Desargues affine plane under Translation: $ \varphi (A + C) = \varphi (A) + \varphi (C)$}
\label{Fig3}
\end{figure}
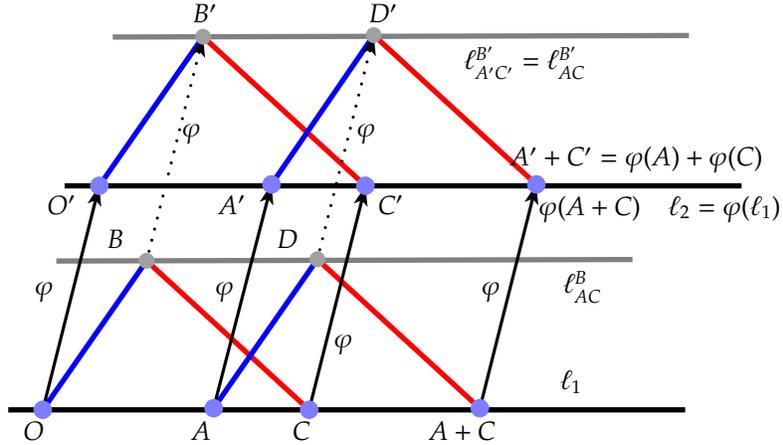

From the additions of points $A$ and $C$ in line $\ell _{1}$, we have that,
the ordered quadruplet \ $(O,A,D,B)$ and $(C,A+C,D,B)$ are parallelograms (see Fig.~\ref{Fig3}).

From the translation properties described at \cite{4}, we have that%
\begin{equation*}
OB\parallel \varphi \left( O\right) \varphi \left( B\right) =O^{\prime
}B^{\prime },BC\parallel \varphi \left( B\right) \varphi \left( C\right)
=B^{\prime }C^{\prime },
\end{equation*}

\begin{equation*}
\varphi \left( O\right) ,\varphi \left( A\right) ,\varphi \left( C\right)
,\varphi (A+C)\in \ell _{2}.
\end{equation*}

For more 
\begin{equation*}
\varphi \left( D\right) \in \ell _{\ell _{2}}^{B^{\prime }}=\ell _{\ell
_{1}}^{B^{\prime }},
\end{equation*}

and have that the following ordered quadruplet are parallelograms:%
\begin{equation*}
(O,B,B^{\prime },O^{\prime });(O,A,A^{\prime },O^{\prime });(B,C,C^{\prime
},B^{\prime });(A,C,C^{\prime },A^{\prime });
\end{equation*}%
\begin{equation*}
(A,D,\varphi \left( D\right) ,A^{\prime });(A,A+C,\varphi (A+C),\varphi
\left( D\right) );(D,A+C,\varphi (A+C),\varphi \left( D\right) ).
\end{equation*}

Now calculate the addition of points $A^{\prime }=\varphi \left( A\right) $
and $C^{\prime }=\varphi \left( C\right) $. From the addition algorithm
(Algorithm 1 in \cite{2} and \cite{3}), we have:%
\begin{equation*}
\forall A^{\prime },C^{\prime }\in \ell _{2},\left[ 
\begin{array}{l}
\mathbf{Step.1.}\exists B^{\prime }=\varphi \left( B\right) \notin \ell _{2}
\\
\mbox{} \\ 
\mathbf{Step.2.\ell }_{\ell _{2}}^{B^{\prime }}\cap \ell _{O^{\prime
}B^{\prime }}^{A^{\prime }}=D^{\prime } \\ 
\mbox{} \\
\mathbf{Step.3.\ell }_{B^{\prime }C^{\prime }}^{D^{\prime }}\cap \ell _{2}=E%
\end{array}%
\right] \Longleftrightarrow A^{\prime }+C^{\prime }=E
\end{equation*}

From the construction of the $A^{\prime }+C^{\prime }$ point, we have that:
\ 
\begin{equation*}
\ell ^{O^{\prime }B^{\prime }}\parallel \ell ^{A^{\prime }D^{\prime }};\ell
^{B^{\prime }C^{\prime }}\parallel \ell ^{D^{\prime }(A^{\prime }+C^{\prime
})};
\end{equation*}

and

\begin{eqnarray*}
D^{\prime } &\in &\mathbf{\ell }_{\ell _{2}}^{B^{\prime }}=\mathbf{\ell }%
_{\ell _{1}}^{B^{\prime }}\text{ \ and } \\
D^{\prime } &\in &\ell _{O^{\prime }B^{\prime }}^{A^{\prime }}=\ell
_{OB}^{A^{\prime }}=\ell _{AD}^{A^{\prime }}=\ell _{AD}^{\varphi \left(
A\right) } \\
\Rightarrow D^{\prime } &=&\varphi \left( D\right) ,
\end{eqnarray*}

and also 
\begin{equation*}
O^{\prime },A^{\prime },C^{\prime },A^{\prime }+C^{\prime }\in \ell
_{2};B^{\prime },D^{\prime }\in \ell _{\ell _{2}}^{B^{\prime }}=\ell _{\ell
_{1}}^{B^{\prime }},
\end{equation*}

therefore%
\begin{equation}
D^{\prime }(A^{\prime }+C^{\prime })\parallel BC,
\end{equation}

but $BC\parallel D(A+C),$ then 
\begin{equation*}
D(A+C)\parallel D^{\prime }(A^{\prime }+C^{\prime }).
\end{equation*}

From here have that, the ordered quadruplet, $(O^{\prime },A^{\prime
},D^{\prime },B^{\prime })$ and $\left( C^{\prime },A^{\prime }+C^{\prime
},D^{\prime },B^{\prime }\right) $ are parallelograms:

From the Desargues axiom (see \cite{2},\cite{3}) and from the above
parallelograms, concretely from the parallelograms:%
\begin{equation*}
\left( A,A+C,\varphi \left( A+C\right) ,A^{\prime }\right) \text{\ and\ }%
\left( A,D,D^{\prime },A^{\prime }\right) ,
\end{equation*}

have that, the ordered quadruplet 
\begin{equation*}
\left( D,A+C,\varphi \left( A+C\right) ,D^{\prime }\right) -\text{is
Parallelogram}.
\end{equation*}

From here have that:%
\begin{equation*}
D(A+C)\parallel D^{\prime }\varphi \left( A+C\right)
\end{equation*}

adding the fact that

\begin{equation*}
A^{\prime }+C^{\prime }=\varphi (A)+\varphi (C)\in \ell _{2}
\end{equation*}

and 
\begin{equation*}
\varphi \left( A+C\right) \in \ell _{2},
\end{equation*}

we have that%
\begin{equation*}
\varphi \left( A+C\right) =\varphi (A)+\varphi (C).
\end{equation*}

$\mathbf{Case 2}$ We see now that stands for multiplication properties: 
\begin{equation*}
\varphi \left( A\ast C\right) =\varphi \left( A\right) \ast \varphi \left(
C\right)
\end{equation*}

During the multiplication of points $A$ and $C,$ we choose a point $I$ 
(necessarily different from the point $O$ that we have assigned in the role 
of the element 'zero'. We do this freely, as each line in affine plan having at least 2-points)
in
the role of "\textbf{one}" (see Fig.~\ref{Fig4}), the translation $\varphi $\ \ gives
us the point $I$ in point $\varphi (I)=I^{\prime },$ according to the
multiplication algorithm, for points $A,C\in \ell _{1}$, we have%
\begin{equation*}
\forall A,C\in \ell _{1},\left\{ 
\begin{array}{l}
\mathbf{Step.}1.\exists B\notin \ell _{1} \\ 
\mbox{} \\
\mathbf{Step.}2.\ell _{IB"}^{A}\cap OB=D \\ 
\mbox{} \\
\mathbf{Step.}3.\ell _{BC}^{D}\cap \ell _{1}=E%
\end{array}%
\right\} \Longleftrightarrow A\ast C=E.
\end{equation*}

From this we have parallelisms

\begin{equation*}
IB\parallel AD,BC\parallel D(A\ast C),
\end{equation*}

and $A\ast C\in \ell _{1}$ (see Fig.~\ref{Fig4}).

\begin{figure}[!ht]
\centering%
\newrgbcolor{ududff}{0.30196078431372547 0.30196078431372547 1.}
\newrgbcolor{xdxdff}{0.49019607843137253 0.49019607843137253 1.}
\newrgbcolor{yqqqyq}{0.5019607843137255 0. 0.5019607843137255}
\newrgbcolor{ffxfqq}{1. 0.4980392156862745 0.}
\psset{xunit=1.0cm,yunit=1.0cm,algebraic=true,dimen=middle,dotstyle=o,dotsize=5pt 0,linewidth=1.6pt,arrowsize=3pt 2,arrowinset=0.25}
\begin{pspicture*}(0.5,0.2)(11.5,6.5)
\psline[linewidth=2.pt](1.,1.)(10.,1.)
\psline[linewidth=2.pt,linecolor=yqqqyq](3.,1.)(3.,2.)
\psline[linewidth=2.pt,linecolor=yqqqyq](4.,1.)(4.032,2.516)
\psline[linewidth=2.pt,linecolor=ffxfqq](3.,2.)(5.,1.)
\psline[linewidth=2.pt,linecolor=ffxfqq](4.032,2.516)(7.18,1.)
\psline[linewidth=2.pt](1.78,4.48)(10.78,4.48)
\psline[linewidth=2.pt,linecolor=yqqqyq](3.78,4.48)(3.78,5.48)
\psline[linewidth=2.pt,linecolor=yqqqyq](4.78,4.48)(4.812,5.996)
\psline[linewidth=2.pt,linecolor=ffxfqq](3.78,5.48)(5.78,4.48)
\psline[linewidth=2.pt,linecolor=ffxfqq](4.812,5.996)(7.96,4.48)
\psline[linewidth=1.2pt]{->}(1.,1.)(1.78,4.48)
\psline[linewidth=1.2pt]{->}(3.,1.)(3.78,4.48)
\psline[linewidth=1.2pt]{->}(4.,1.)(4.78,4.48)
\psline[linewidth=1.2pt]{->}(5.,1.)(5.78,4.48)
\psline[linewidth=1.2pt]{->}(7.18,1.)(7.96,4.48)
\psline[linewidth=1.2pt,linestyle=dashed,dash=1pt 1pt]{->}(3.,2.)(3.78,5.48)
\psline[linewidth=1.2pt,linestyle=dashed,dash=1pt 1pt]{->}(4.032,2.516)(4.812,5.996)
\rput[tl](0.82,0.85){$O$}
\rput[tl](2.86,0.85){$I$}
\rput[tl](3.86,0.85){$A$}
\rput[tl](4.88,0.85){$C$}
\rput[tl](6.82,0.85){$A*C$}
\rput[tl](1.22,4.3){$O'$}
\rput[tl](3.88,4.3){$I'$}
\rput[tl](4.9,4.3){$A'$}
\rput[tl](5.98,4.3){$C'$}
\rput[tl](7.96,4.3){$\varphi(A*C)$}
\rput[tl](3.26,5.8){$B'$}
\rput[tl](4.36,6.3){$D'$}
\rput[tl](2.42,2.2){$B$}
\rput[tl](3.72,2.9){$D$}
\rput[tl](5.78,3.4){$\ell^{OB}$}
\rput[tl](8.16,1.4){$\ell_{1}$}
\rput[tl](9.7,4.33){$\ell_{2}=\varphi(\ell_{1})$}
\rput[tl](7.84,5.){$A'*C'=\varphi(A)*\varphi(B)$}
\rput[tl](5.56,6.3){$\ell^{O'B'}$}
\rput[tl](0.82,2.92){$\varphi$}
\psline[linewidth=2.pt](1.78,4.48)(5.716,6.448)
\psline[linewidth=2.pt](1.,1.)(5.824,3.412)
\begin{scriptsize}
\psdots[dotsize=6pt 0,dotstyle=*,linecolor=ududff](1.,1.)
\psdots[dotsize=6pt 0,dotstyle=*,linecolor=xdxdff](3.,1.)
\psdots[dotsize=6pt 0,dotstyle=*,linecolor=xdxdff](4.,1.)
\psdots[dotsize=6pt 0,dotstyle=*,linecolor=xdxdff](5.,1.)
\psdots[dotstyle=*,linecolor=gray](3.,2.)
\psdots[dotstyle=*,linecolor=gray](4.032,2.516)
\psdots[dotsize=6pt 0,dotstyle=*,linecolor=xdxdff](7.18,1.)
\psdots[dotsize=6pt 0,dotstyle=*,linecolor=ududff](1.78,4.48)
\psdots[dotsize=6pt 0,dotstyle=*,linecolor=xdxdff](3.78,4.48)
\psdots[dotsize=6pt 0,dotstyle=*,linecolor=xdxdff](4.78,4.48)
\psdots[dotsize=6pt 0,dotstyle=*,linecolor=xdxdff](5.78,4.48)
\psdots[dotstyle=*,linecolor=gray](3.78,5.48)
\psdots[dotstyle=*,linecolor=gray](4.812,5.996)
\psdots[dotsize=6pt 0,dotstyle=*,linecolor=xdxdff](7.96,4.48)
\end{scriptsize}
\end{pspicture*}
\caption{The multiplication of two points in a line on the Desargues affine plane under Translation: $ \varphi (A * C) = \varphi (A) * \varphi (C)$}
\label{Fig4}
\end{figure}
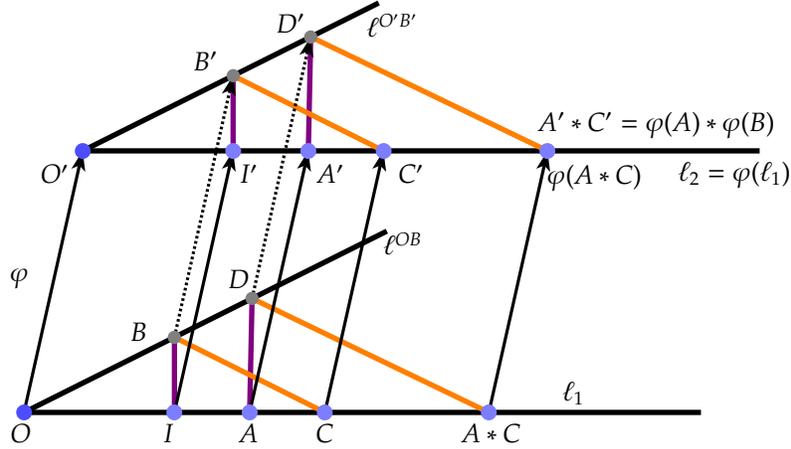

By translation properties have: 
\begin{equation*}
O^{\prime }\in \ell _{2},I^{\prime }\in \ell _{2},C^{\prime }\in \ell
_{2},A^{\prime }\in \ell _{2},\varphi \left( A\ast C\right) \in \ell _{2}
\end{equation*}

also for these points we have:%
\begin{equation*}
OO^{\prime }\parallel II^{\prime }\parallel AA^{\prime }\parallel CC^{\prime
}\parallel \left( A\ast C\right) \varphi \left( A\ast C\right) .
\end{equation*}

Calculate now $A^{\prime }\ast C^{\prime }$ according to multiplication
Algorithm (see \cite{2},\cite{3}). During the multiplication of points $A$
and $C$ we chose, as auxiliary points, the point $B\notin \ell _{1},$ here
we choose the point $B^{\prime }=\varphi \left( B\right) \notin \ell _{2}$
(we do this for ease of proof, as in \cite{2}, we have shown that the choice
of auxiliary point $B$, or in our case $\varphi (B)$ is arbitrary), by
applying the multiplication algorithm we have:%
\begin{equation*}
\forall A^{\prime }=\varphi \left( A\right) ,C^{\prime }=\varphi \left(
C\right) \in \ell _{2},\left\{ 
\begin{array}{l}
\mathbf{Step.}1.\exists B^{\prime }=\varphi \left( B\right) \notin \ell _{2}
\\ 
\mbox{} \\
\mathbf{Step.}2.\ell _{I^{\prime }B^{\prime }}^{A^{\prime }}\cap O^{\prime
}B^{\prime }=D^{\prime } \\ 
\mbox{} \\
\mathbf{Step.}3.\ell _{B^{\prime }C^{\prime }}^{D^{\prime }}\cap \ell
_{2}=E^{\prime }%
\end{array}%
\right\} \Longleftrightarrow A^{\prime }\ast C^{\prime }=E^{\prime }.
\end{equation*}

From this we have parallelisms 
\begin{equation*}
I^{\prime }B^{\prime }\parallel A^{\prime }D^{\prime },B^{\prime }C^{\prime
}\parallel D^{\prime }(A^{\prime }\ast C^{\prime }),
\end{equation*}

but from above, we have that 
\begin{equation*}
IB\parallel I^{\prime }B^{\prime },AD\parallel A^{\prime }D^{\prime
},BC\parallel B^{\prime }C^{\prime },
\end{equation*}

then 
\begin{equation*}
D^{\prime }\in \ell _{I^{\prime }B^{\prime }}^{A^{\prime }}=\ell
_{IB}^{A^{\prime }}=\ell _{AD}^{A^{\prime }},
\end{equation*}

also 
\begin{equation*}
D^{\prime }\in O^{\prime }B^{\prime }=\varphi (OB)=\ell ^{\varphi (O)\varphi
(B)}.
\end{equation*}

Then we have the ordered quadruplet \ 

\begin{equation*}
(B,D,D^{\prime },B^{\prime })-\text{is Parallelogram,}
\end{equation*}

but by the translation, we have

\begin{equation*}
(B,D,\varphi \left( D\right) ,\varphi \left( B\right) )=(B,D,\varphi \left(
D\right) ,B^{\prime })-\text{is a Parallelogram.}
\end{equation*}

So we have two similar parallelogams that have the three same points.  
Consequently, these two parallelograms will be the same, {|em i.e.},

\begin{equation*}
(B,D,D^{\prime },B^{\prime })=(B,D,\varphi \left( D\right) ,B^{\prime }).
\end{equation*}

Hence, 
\begin{equation*}
D^{\prime }=\varphi \left( D\right) .
\end{equation*}

%

Thus \ 
\begin{equation*}
A^{\prime }\ast C^{\prime }\in \ell _{2}
\end{equation*}

and \ 
\begin{equation*}
A^{\prime }\ast C^{\prime }\in \ell _{B^{\prime }C^{\prime }}^{D^{\prime
}}=\ell _{BC}^{D^{\prime }}=\ell _{D(A\ast C)}^{\varphi \left( D\right) },
\end{equation*}

implies%
\begin{equation*}
A^{\prime }\ast C^{\prime }=\ell _{D(A\ast C)}^{\varphi \left( D\right)
}\cap \ell _{2}.
\end{equation*}

From the definition and the properties of translation, we also have,%
\begin{equation*}
A\ast C\in \ell _{1}\Rightarrow \varphi \left( A\ast C\right) \in \ell _{2},
\end{equation*}

So, we have that the ordered quadruplet \ 

\begin{equation*}
(D,A\ast C,\varphi \left( A\ast C\right) ,D^{\prime })-is\text{ \ }%
Parallelogram,
\end{equation*}

and we also have (from construction of the points $A^{\prime }\ast C^{\prime }$) the
ordered quadruplet

\begin{equation*}
(D,A\ast C,A^{\prime }\ast C^{\prime },D^{\prime })=(D,A\ast C,A^{\prime
}\ast C^{\prime },\varphi \left( D\right) )-is\text{ \ }Parallelogram.
\end{equation*}

and
\begin{eqnarray*}
\varphi \left( A\ast C\right) &\in &\ell _{D(A\ast C)}^{\varphi \left(
D\right) }\Rightarrow \\
\varphi \left( A\ast C\right) &=&\ell _{D(A\ast C)}^{\varphi \left( D\right)
}\cap \ell _{2} \\
&=&A^{\prime }\ast C^{\prime } \\
&=&\varphi \left( A\right) \ast \varphi \left( C\right) .
\end{eqnarray*}

Hence, we obtain the desired result, namely,%
\begin{equation*}
\varphi \left( A\ast C\right) =\varphi \left( A\right) \ast \varphi \left(
C\right) .
\end{equation*}
\end{proof}

From Lemma~\ref{lemma:one} and Lemma~\ref{lemma:two}, we obtain the main result, namely,

\begin{Theorem}\label{thm:mainResult}
Each dilatation in a finite Desargues affine plane $\mathcal{A}_{\mathcal{D}}=(%
\mathcal{P},\mathcal{L},\mathcal{I})$ is an isomorphism's between
skew-fields constructed over isomorphic lines $\ell _{1},\ell _{2}\in 
\mathcal{L},$\ of that plane.
\end{Theorem}

\begin{Theorem}
Each dilatation in a finite Desargues affine plane $\mathcal{A}_{\mathcal{D}}=(%
\mathcal{P},\mathcal{L},\mathcal{I})$ occurs in a Pappian space.
\end{Theorem}
\begin{proof}
Immediate from Theorem~\ref{thm:Tecklenburg}.
\end{proof}

\bibliographystyle{amsplain}
\bibliography{NSrefs}

\providecommand{\bysame}{\leavevmode\hbox to3em{\hrulefill}\thinspace}
\providecommand{\MR}{\relax\ifhmode\unskip\space\fi MR }
\providecommand{\MRhref}[2]{%
  \href{http://www.ams.org/mathscinet-getitem?mr=#1}{#2}
}
\providecommand{\href}[2]{#2}
\begin{thebibliography}{10}

\bibitem{ADAleksandrov1956nonEuclideanGeometry}
A.D. Aleksandrov, \emph{Non-euclidean geometry}, Mathematics. {I}ts Contents,
  Methods and Meaning (M.A.~Lavrent'ev A.D.~Aleksandrov, A.N.~Kolmogorov, ed.),
  Dover Pubs, Inc., Mineola, NY, 1956,1999, Trans. from the Russian edition by
  S.H. Gould, pp.~97--192.

\bibitem{PSAleksandrov1956affineGeometry}
P.S. Aleksandrov, \emph{Topology}, Mathematics. {I}ts Contents, Methods and
  Meaning (M.A.~Lavrent'ev A.D.~Aleksandrov, A.N.~Kolmogorov, ed.), Dover Pubs,
  Inc., Mineola, NY, 1956,1999, Trans. from the Russian edition by S.H. Gould,
  pp.~193--226.

\bibitem{Berger2010geometryRevealed}
M.~Berger, \emph{Geometry revealed}, Springer, Heidelberg, 2010, {x}vi+831 pp.,
  ISBN: 978-3-540-70996-1, MR2724440.

\bibitem{7}
H.~S.~M. Coxeter, \emph{Introduction to geometry, 2nd ed.}, John Wiley \& Sons,
  Inc., New York-London-Sydney, 1969, {x}vii+469 pp., MR0123930, MR0346644.

\bibitem{DolgachevShirokov1995affineMap}
I.V. Dolgachev and A.P. Shirokov, \emph{Affine space}, Encyclopedia of
  Mathematics (M.~Hazewinkel, ed.), Kluwer, Dordrecht, 1995, pp.~62--63.

\bibitem{3}
K.~Filipi, O.~Zaka, and A.~Jusufi, \emph{The construction of a corp in the set
  of points in a line of desargues affine plane}, Matematicki Bilten
  \textbf{43} (2019), no.~01, 1--23, ISSN 0351-336X (print), ISSN 1857--9914
  (online).

\bibitem{Hilbert1959geometry}
D.~Hilbert, \emph{The foundations of geometry}, The Open Court Publishing Co.,
  La Salle, Ill., 1959, {v}ii+143 pp., MR0116216.

\bibitem{Kryftis2015thesis}
A.~Kryftis, \emph{A constructive approach to affine and projective planes},
  Ph.D. thesis, University of Cambridge, Trinity College and Department of Pure
  Mathematics and Mathematical Statistics, 2015, supervisor: M. Hyland,
  v+170pp.,arXiv 1601.04998v1 19 Jan. 2016.

\bibitem{Wedderburn1905DesarguesianAffinePlane}
J.H. Maclagan-Wedderburn, \emph{A theorem on finite algebras}, Trans. Amer.
  Math. Soc. \textbf{6} (1905), no.~3, 349--352.

\bibitem{NollSchaffer1977affineIsomorphisms}
W.~Noll and J.J. Sch\"{a}ffer, \emph{Order-isomorphisms in affine spaces},
  Annali di Matematica Pura ed Applicata \textbf{117} (1978), no.~4, 243--262,
  MR0515964.

\bibitem{Peters2016AMSJphysicalGeometry}
J.F. Peters, \emph{Two forms of proximal, physical geometry. {A}xioms, sewing
  regions together, classes of regions, duality and parallel fibre bundles},
  Advan. in Math: Sci. J \textbf{5} (2016), no.~2, 241--268, Zbl 1384.54015,
  reviewed by D. Leseberg, Berlin.

\bibitem{Peters2018JMSMvortexNerves}
\bysame, \emph{Proximal vortex cycles and vortex nerve structures.
  non-concentric, nesting, possibly overlapping homology cell complexes},
  Journal of Mathematical Sciences and Modelling \textbf{1} (2018), no.~2,
  56--72, ISSN 2636-8692, \url{https://dx.doi.org/10.33187/jmsm.425066}, See,
  also, \url{https://arxiv.org/abs/1805.03998}.

\bibitem{Pickert1973PlayfairAxiom}
G.~Pickert, \emph{Affine planes: {A}n example of research on geometric
  structures}, The Mathematical Gazette \textbf{57} (2004), no.~402, 278--291,
  MR0474017.

\bibitem{Prazmowska2004DemoMathDesparguesAxiom}
M.~Pra\.{z}mowska, \emph{A proof of the projective {D}esargues axiom in the
  {D}esarguesian affine plane}, Demonstratio Mathematica \textbf{37} (2004),
  no.~4, 921--924, MR2103894.

\bibitem{Szmielew1981DesarguesAxiom}
W.~Szmielew, \emph{Od geometrii afinicznej do euklidesowej (polish) [from
  affine geometry to euclidean geometry] rozwa?ania nad aksjomatyk? [an
  approach through axiomatics]}, Biblioteka Matematyczna [Mathematics Library],
  Warsaw, 1981, 172 pp., ISBN: 83-01-01374-5, MR0664205.

\bibitem{Wedderburn1987PappusDesarguesianAffinePlanes}
H.~Tecklenburg, \emph{A proof of the theorem of {P}appus in finite
  {D}esarguesian affine planes}, Journal of Geometry \textbf{30} (1987),
  173--181.

\bibitem{8}
O.~Zaka, \emph{Contribution to reports of some algebraic structures with affine
  plane geometry and applications}, Ph.D. thesis, Polytechnic University of
  Tirana,Tirana, Albania, Department of Mathematical Engineering, 2016,
  supervisor: K. Filipi, vii+113pp.

\bibitem{4}
\bysame, \emph{A description of collineations-groups of an affine plane},
  Libertas Mathematica (N.S.) \textbf{37} (2017), no.~2, 81--96, ISSN print:
  0278 -- 5307, ISSN online: 2182 -- 567X, MR3828328.

\bibitem{5}
\bysame, \emph{Three vertex and parallelograms in the affine plane: Similarity
  and addition abelian groups of similarly $n$-vertexes in the {D}esargues
  affine plane}, Mathematical Modelling and Applications \textbf{3} (2018),
  no.~1, 9--15, \url{http://doi:10.11648/j.mma.20180301.12}.

\bibitem{1}
\bysame, \emph{Dilations of line in itself as the automorphism of the
  skew-field constructed over in the same line in desargues affine plane},
  Applied Mathematical Sciences \textbf{13} (2019), no.~5, 231--237.

\bibitem{2}
O.~Zaka and K.~Filipi, \emph{The transform of a line of {D}esargues affine
  plane in an additive group of its points}, Int. J. of Current Research
  \textbf{8} (2016), no.~07, 34983--34990.

\end{thebibliography}

\end{document}